\documentclass[12pt,oneside]{amsproc}

\usepackage{amssymb}
\usepackage{amsmath}
\usepackage[T2A]{fontenc}
\usepackage[cp1251]{inputenc}
\usepackage[russian]{babel}
\usepackage{epic,eepic}

\numberwithin{equation}{section}

\newtheorem{thm}{Теорема}[section]
\newtheorem{stat}{Утверждение}[section]

\newtheorem{rim}{Замечание}[section]

\newcommand{\sgn}{\operatorname{sgn}}

\newcommand{\esssup}{\mathop{\rm ess\,sup}\limits}

\title[]{Точные оценки производных высокого порядка в пространствах Соболева}
\author{Т.А.~Гарманова, И.А.~Шейпак}
\address{Московский государственный университет
им.~М.~В.~Ломоносова, механико-математический факультет, Московский центр фундаментальной
и прикладной математики}
\email{garmanovata@gmail.com, iasheip@yandex.ru}
\thanks{Результаты разделов 2 и 3 получены при  поддержке  фонда <<Базиc>> (<<Стипендии — Механико-математический факультет>> для аспирантов Механико-математического факультета МГУ им. М.В. Ломоносова -- 2021), результаты разделов 4 и 5 получены при  поддержке РНФ грант №~20-11-20261}

\begin{document}
\noindent УДК~517.518.23, MSC 2010 Classification Codes 26D10, 46E35, 46E22
\begin{abstract}
 В работе получено описание сплайнов $Q_{n,k}(x,a)$, которые для произвольной точки $a\in(0;1)$ и произвольной функции $y\in\mathring{W}^n_p[0;1]$ задают соотношения $y^{(k)}(a)=\int_0^1 y^{(n)}(x)Q^{(n)}_{n,k}(x,a)dx$. Указана связь задачи о минимизации нормы $Q^{(n)}_{n,k}$ в $L_{p'}[0;1]$ ($1/p+1/p'=1$) с задачей о наилучших оценках производных $y^{(k)}(a)\leqslant A_{n,k,p}\|y^{(n)}\|_{L_p[0;1]}$, а также  c задачей нахождения точных констант вложения пространства Соболева   $\mathring{W}^n_p[0;1]$ в пространство $\mathring{W}^k_\infty[0;1]$, $n\in\mathbb{N}$,  $k=0,1,\ldots, n-1$. Найдены точные константы вложения при $k=n-1$ и $p=\infty$, а также при всех  $n\in\mathbb{N}$,  $k=0,1,\ldots, n-1$ и $p=1$.
\end{abstract}
\maketitle

Ключевые слова: оценки производных, пространства Соболева, теоремы вложения

\section{Введение}

Рассмотрим пространство Соболева $\mathring{W}^n_p[0;1]$, состоящее из вещественнозначных функций, все производные которых до $n-1$-го порядка абсолютно непрерывны,  $f^{(n)}\in L_p[0;1]$, а также удовлетворяющих краевым условиям $y^{(j)}(0)=y^{(j)}(1)=0$, $j=0,1,\ldots,n-1$. Это пространство снабжено естественной нормой
$$
\|y\|_{\mathring{W}^n_p[0;1]}:=\left(\int_0^1|y^{(n)}(x)|^pdx\right)^{1/p}.
$$
В случае $p=\infty$ норма определяется формулой
$$
\|y\|_{\mathring{W}^n_\infty[0;1]}:=\esssup_{x\in[0;1]}|y^{(n)}(x)|
$$

Во многих задачах теории функций, функционального анализа, математической физики и в других областях возникает вопрос о нахождении точных констант $\Lambda_{n,k,p,q}$
\begin{equation}\label{eq:La}
\Lambda_{n,k,p,q}:=\sup\{\|y^{(k)}\|_{L_q[0;1]}: \; \|y^{(n)}\|_{L_p[0;1]}=1\},
\end{equation}
которые  равны норме операторов вложения пространств $\mathring{W}^n_p[0;1]$  в пространства $\mathring{W}^k_q[0;1]$, $n\in\mathbb{N}$,  $k=0,1,\ldots,n-1$. Кроме того, в задачах о нахождении $\Lambda_{n,k,p,q}$ нередко возникают специальные функций и тем самым служат развитию теории специальных функций.

Для $n=1$ и $k=0$ эта константа вычислена в \cite{Shm} при всех значениях $p$ и $q$. При $n=2$, $k=1$  и $q\leqslant 3p$ эти константы получены в \cite{Naz}. Вычисление этих констант для произвлдьных $p,q\geqslant 1$, $n\in\mathbb{N}$, $k=0,1,\ldots,n-1$ оказалось сложной задачей. Отдельный интерес представляет задача о нахождении $\Lambda_{n,k,p,q}$ при $q=\infty$. В этом случае возникает дополнительная задача о нахождении точных величин $A_{n,k,p}(a)$ в неравенствах

\begin{equation}\label{eq:A}
|y^{(k)}(a)|\leqslant A_{n,k,p}(a)\cdot\|y^{(n)}\|_{L_p[0;1]}.
\end{equation}

Тогда точную   константу вложения пространства  $\mathring{W}^n_p[0;1]$ в пространство $\mathring{W}^k_\infty[0;1]$ можно определить как
$$
\Lambda_{n,k,p,\infty}=\sup_{a\in[0;1]}A_{n,k,p}(a).
$$
Здесь наиболее изучен случай $p=2$. В работе \cite{Kalyab} получено представление для $A^2_{n,k,2}(a)$ в виде ряда по первообразным полиномов Лежандра и получены значения $\Lambda_{n,k,2,\infty}$ при $k=0,1,2$. В работах \cite{NazM}--\cite{NazM2} получены константы $\Lambda_{n,k,2,\infty}$ при $k=4,6$.  В \cite{Garm} для  величин $A^2_{n,k,2}(a)$ найдено представление через гипергеометрические функции типа ${}_3F_2$.
Также в работе \cite{Kalyab} А.Г.~Калябиным была выдвинута гипотеза, что при четных $k$ глобальный максимум функции $A^2_{n,k,2}$ достигается в середине отрезка, а при нечетных $k$ функция $A^2_{n,k,2}$ имеет в середине отрезка локальный минимум. В работе \cite{NazM} доказано, что при четных $k$ середина отрезка является для функций $A^2_{n,k,2}$ локальным максимумом, а при нечетных --- локальным минимумом. В работе \cite{GarmSh2} методом построения огибающей гипотеза А.Г.~Калябина полностью доказана. Более того, показано, что при нечетных $k$ точкой глобального максимума функции $A_{n,k,2}$ является точка локального максимума, ближайшая к середине отрезка.

В этой работе мы будем изучать свойства сплайнов $Q_{n,k}(x,a)$ которые задают в пространствах $\mathring{W}^n_p[0;1]$ функционалы $f\mapsto f^{(k)}$ ($a\in(0;1)$). Задача нахождения величин $A_{n,k,p}(a)$ непосредственно связана с минимизацией нормы сплайнов $Q^{(n)}_{n,k}(x,a)$ в $L_{p'}[0;1]$ ($1/p+1/p'=1$).  В свой очередь задача о минимизации $\|Q^{(n)}_{n,k}\|_{L_{p'}[0;1]}$ приводит к задаче о наилучшем приближении сплайна
$$
S_{n,k}(x,a)=\left\{\begin{aligned}
&\frac{(-1)^{n-k-1}(x-a)^{n-k-1}}{(n-k-1)!},\quad x\in[0;a],\\
&0,\quad x\in[a;1]\end{aligned}\right.
$$
алгебраическими многочленами. Здесь отдельное место занимает случай $k=n-1$ (сплайн $S_{n,n-1}(x,a)=\chi_{[0;a]}$ разрывен и равен характеристической функции отрезка $[0;a]$). В случае  $p=\infty$ также удается построить огибающую для $A_{n,n-1,\infty}(a)$, для которой, как и для $p=2$, оказывается верной гипотеза А.Г.Калябина (см. раздел 4).  Также рассматривается и другой предельный случай $p=1$, в котором максимум $A_{n,k,1}(a)$ не зависит от точки $a$. Задача о приближении характеристической функции тригонометрическими многочленами решена в \cite{BK}.

Структура работы следующая. В разделе 2  получено интегральное представления функционала $y\mapsto y^{(k)}(a)$ для функций $y\in\mathring{W}^n_p[0;1]$. В разделе 3
установлена связь задачи о нахождении наилучших величин $A_{n,k,p}(a)$  в неравенствах \eqref{eq:A} c нахождением сплайнов специального вида, наименее уклоняющихся от нуля в $L_{p'}[0;1]$. В разделе 4  на основе свойств сплайнов $Q^{(n)}_{n,k}(x,a)$ получены точные константы вложения $\Lambda_{n,n-1,\infty,\infty}$. В разделе 5 получены точные константы вложения $\Lambda_{n,k,1,\infty}$  при всех $n\in\mathbb{N}$ и $k=0,1,\ldots,n-1$.

\section{Функционал означивания}\label{sec:2}

В этом разделе мы получим интегральное представление функционалов $y\mapsto y^{(k)}(a)$ в $\mathring{W}^n_p[0;1]$ и  изучим некоторые свойства функций $Q_{n,k}$, для которых
$$
y^{(k)}(a)=\int_{0}^1y^{(n)}Q^{(n)}_{n,k}(x,a)\,dx.
$$

Рассмотрим функции

\begin{equation}\label{eq:Q}
Q^{(n)}_{n,k}(x,a):=
\left\{\begin{aligned}
&\sum_{l=0,l\ne k}^{n-1}
\frac{(-1)^{n-l-1}(x-a)^{n-1-l}\nu_l}{(n-1-l)!}+\frac{(-1)^{n-1-k}(x-a)^{n-1-k}(\nu_k+1)}{(n-1-k)!},\quad x\in [0;a)\\
&\sum_{l=0}^{n-1}\frac{(-1)^{n-l-1}(x-a)^{n-1-l}\nu_l}{(n-1-l)!}, \quad x\in (a;1],
\end{aligned}
\right.
\end{equation}
где $\nu_l=\nu_l(a)$ --- некоторые коэффициенты, зависящие от $a$, $l=0,1,\ldots,n-1$.

\begin{stat}\label{st:11}
Для произвольной функции $y\in\mathring{W}^n_p[0;1]$ выполнено
\begin{equation}\label{eq:zn_a}
y^{(k)}(a)=\int_{0}^1y^{(n)}Q^{(n)}_{n,k}(x,a)\,dx.
\end{equation}
\end{stat}
\begin{proof}
Справедливость утверждения следует из формулы \eqref{eq:Q} и интегрирования по частям интегралов $\int_{0}^a y^{(n)}Q^{(n)}_{n,k}(x,a)\,dx$ и $\int_{a}^1y^{(n)}Q^{(n)}_{n,k}(x,a)\,dx$.
\end{proof}

Пользуясь представлением \eqref{eq:zn_a} получаем оценку для $k$-ой производной функции $y$:
\begin{equation}\label{eq:Held}
|y^{(k)}(a)|\leqslant \|y^{(n)}\|_{L_p[0;1]}\cdot \|Q^{(n)}_{n,k}\|_{L_{p'}[0;1]}.
\end{equation}

Для получения точной оценки в неравенстве \eqref{eq:A} необходимо решить две задачи: 1) в классе функций вида \eqref{eq:Q} найти функцию $\hat Q_{n,k}$, $n$-ая производная которой имеет наименьшую норму в $L_{p'}[0;1]$;
2) для функции $\hat Q_{n,k}$ построить экстремальную функцию, для которой в неравенстве \eqref{eq:Held} достигается равенство. Покажем, что это связанные друг с другом задачи.

\begin{stat}\label{stat:minQ}
Пусть $1<p<\infty$ и пусть $\hat Q_{n,k}$ --- функция  вида \eqref{eq:Q}, для которой
$$
\|\hat Q^{(n)}_{n,k}\|_{L_{p'}[0;1]}=\min_{\nu_0,\nu_1,\ldots, n_{n-1}}\|Q^{(n)}_{n,k}\|_{L_{p'}[0;1]}.
$$
Тогда она удовлетворяет соотношениям
\begin{equation}\label{eq:orth}
\int_0^1 x^j |\hat Q^{(n)}_{n,k}|^{p'-1}\sgn\hat Q^{(n)}_{n,k}dx=0,\quad j=0,1,\ldots, n-1.
\end{equation}
\end{stat}
\begin{proof}
Так как надо минимизировать норму по параметрам $\nu_0$, $\nu_1$, \ldots, $\nu_{n-1}$, то необходимо выполнены условия
$$
\frac{\partial}{\partial \nu_j}\|\hat Q^{(n)}_{n,k}\|^{p'}_{L_{p'}[0;1]}=0, \quad j=0,1,\ldots,n-1.
$$

Так как
$$
\frac{\partial}{\partial \nu_j}\|\hat Q^{(n)}_{n,k}\|^{p'}_{L_{p'}[0;1]}=\frac{p'(-1)^{n-1-j}}{(n-j-1)!}\int_0^1 (x-a)^j |\hat Q^{(n)}_{n,k}|^{p'-1}\sgn\hat Q^{(n)}_{n,k}dx,
$$
то из равенств $\frac{\partial}{\partial \nu_j}\|\hat Q^{(n)}_{n,k}\|^{p'}_{L_{p'}[0;1]}=0$ при $j=0,1,\ldots,n-1$  следует утверждение леммы.
\end{proof}

\begin{rim}
Условия \eqref{eq:orth} являются и достаточными. Поскольку норма выпуклая функция, то минимум существует.
\end{rim}

\begin{stat}\label{stat:orth1}
Пусть $1<p<\infty$ и пусть функция $h\in L_p[0;1]$ удовлетворяет соотношениям
$$
\int_{0}^1 x^j h(x)\,dx=0,\quad j=0,1,\ldots,n-1.
$$
Тогда существует решение уравнения
$$
y^{(n)}=h(x),
$$
принадлежащее пространству $\mathring{W}^{n}_p[0;1]$.
\end{stat}
\begin{proof}
Положим
$$
y(x)=\int_0^x\frac{(x-t)^{n-1}}{(n-1)!}h(t)\,dt
$$

Краевые условия в нуле выполнены в силу интеграла по отрезку $[0;x]$, а в единице --- в силу свойств функции $h$.
\end{proof}

Введем функцию
\begin{equation}\label{eq:w}
w^{(n)}_{n,k}(x):=\left|\hat Q^{(n)}_{n,k}(x)\right|^{p'-1}\cdot \sgn \hat Q^{(n)}_{n,k}(x).
\end{equation}

Поскольку $\hat Q^{(n)}_{n,k}$ является сплайном, то $w^{(n)}_{n,k}\in L_p[0;1]$. А тогда из утверждений \ref{stat:minQ} и \ref{stat:orth1} следует, что существует ее первообразная $w_{n,k}\in\mathring{W}^n_p[0;1]$, для  которой достигается равенство в неравенстве \eqref{eq:A} с наименьшей возможной величиной $A_{n,k,p}(a)$, т.е. для которой выполнено
$$
w^{(k)}_{n,k}(a)= A_{n,k,p}(a)\cdot\|w^{(n)}_{n,k}\|_{L_p[0;1]}.
$$

\section{Связь задачи о минимизации нормы $Q^{(n)}_{n,k}$ с минимизацией уклонений}

Напомним необходимые сведения о полиномах Лежандра на отрезке $[0;1]$. Полином Лежандра определяется как $P_m(x):=\frac{1}{m!}((x^2-x)^m)^{(m)}$. Первообразную порядка $j\geqslant 0$ определим следующим образом:  $P^{(-j)}_m(x):=\frac{1}{m!}((x^2-x)^m)^{(m-j)}$. Система полиномов Лежандра образует ортогональный базис в $L_2[0;1]$ c нормировкой
 $\|P_m\|^2_{L_2[0;1]}=\frac{1}{2m+1}$.

Задача о нахождении $\min_{\nu_0,\nu_1,\ldots, n_{n-1}}\|Q^{(n)}_{n,k}\|_{L_{p'}[0;1]}$ тесно связана со следующей задачей. Известно (см. \cite{GarmSh}), что при $p=2$ функционал $f\mapsto f^{(k)}(a)$ в $\mathring{W}^n_2[0;1]$ задан в соответствии с теоремой Рисса единственной функцией $g_{n,k}(x,a)$, которая  имеет вид
\begin{equation}\label{eq:g}
g_{n,k}(x,a)=\left\{\begin{aligned}
&\dfrac{(-1)^{n-k-1}}{(2n-k-1)!}(1-a)^{n-k}x^n h_{n,k}(1-x,1-a),\quad x\in [0;a]\\
&\dfrac{(-1)^{n-1}}{(2n-k-1)!}a^{n-k}(1-x)^{n} h_{n,k}(x,a), \quad x\in [a;1],
\end{aligned}\right.
\end{equation}
где
$$
h_{n,k}(x,a)=\sum_{l=0}^{n-1}(-1)^{n-1-l}C_{2n-1-k}^{n-1-l}x^{n-1-l}a^l\sum_{m=0}^{l}C_{n-1+m}^{m}x^{m}.
$$

Т.е.
$$
f^{(k)}(a)=\int_0^1 f^{(n)}(x)g^{(n)}_{n,k}(x,a)dx.
$$

\begin{thm}\label{thm:gQ}
Функции $\hat Q^{(n)}_{n,k}$ и $g^{(n)}_{n,k}$ отличаются на многочлен степени $n-1$, т.е.
$$
\hat Q^{(n)}_{n,k}-g^{(n)}_{n,k}=\sum_{l=0}^{n-1} c_l(a)x^l.
$$
\end{thm}
\begin{proof}

1) Из вида \eqref{eq:Q} следует, что $Q^{(n)}_{n,k}\in L_2[0;1]$, поэтому $Q^{(n)}_{n,k}$ можно разложить в ряд по ортогональной системе полиномов Лежандра:
$$
Q^{(n)}_{n,k}(x,a)=\sum_{m=0}^\infty \alpha_m(a) P_m(x).
$$

Из свойств функций \eqref{eq:Q} и свойств ортогональности многочленов Лежандра следует, что
$$
\int_0^1 Q^{(n)}_{n,k} P_m(x)dx=\int_0^1 Q^{(n)}_{n,k} \left(P^{(-n)}_m(x)\right)^{(n)}dx=P^{(k-n)}_m(a)
$$
С другой стороны, при  $m\geqslant n$ первообразные полиномов Лежандра $P^{(-n)}_m$ удовлетворяют условиям Дирихле, а поэтому ортогональны всем степеням $x^j$ при $j=0,1,\ldots, n-1$, а следовательно и всем полиномам Лежандра $P_m$ при $m<n$.  Следовательно, при $m\geqslant n$ имеем
$$
\alpha_m(a)=(2m+1)P^{(k-n)}_m(a).
$$

Таким образом,
$$
Q^{(n)}_{n,k}(x,a)=\sum_{m=0}^{n-1} \alpha_m P_m(x)+\sum_{m=n}^{\infty} (2m+1)P^{(k-n)}_m(a) P_m(x).
$$

2) Функции $g_{n,k}$ удовлетворяют условиям Дирихле, поэтому разложение $g^{(n)}_{n,k}$ в ряд по полиномам Лежандра имеет вид
$$
g^{(n)}_{n,k}(x,a)=\sum_{m=n}^\infty \beta_m(a) P_m(x).
$$

Как и пункте 1) доказательства получаем, что $\beta_m=(2m+1)P^{(k-n)}_m(a)$, откуда и следует утверждение теоремы.
\end{proof}

Определим $\mathcal{P}_m$ --- пространство вещественных многочленов
$$
\mathcal{P}_m=\left\{\sum_{j=0}^m c_jx^j,\quad x,c_j\in\mathbb{R}, j=0,1,\ldots,m\right\}
$$
степени не выше $m$.

Таким образом, теорема \ref{thm:gQ} утверждает, что
\begin{equation}\label{eq:ming}
\min_{\nu_0,\nu_1,\ldots,\nu_{n-1}}\|Q^{(n)}_{n,k}\|_{L_{p'}[0;1]}=\min_{u\in\mathcal{P}_{n-1}}\|g^{(n)}_{n,k}-u\|_{L_{p'}[0;1]}.
\end{equation}

\section{Точные константы вложения при $k=n-1$, $p=\infty$}

Рассмотрим следующие функции
\begin{gather*}
g_1(x,a):=g^{(n)}_{n,n-1}(x,a)\quad x\in[0;a],\\
g_2(x,a):=g^{(n)}_{n,n-1}(x,a)\quad x\in(a;1].
\end{gather*}

Из формулы \eqref{eq:g} следует, что $g_1(x,a)-g_2(x,a)=1$, т.е. задача минимизации \eqref{eq:ming} при $k=n-1$, $p=\infty$ равносильна следующей экстремальной задаче:
$$
\|\chi_{[0;a]}-u\|_{L_1[0;1]}\to \min,\quad u\in\mathcal{P}_{n-1}.
$$

 При этом, из определения функций $A_{n,k,p}$ следует, что
\begin{equation}\label{eq:gA}
A_{n,k,p}(a)=\min_{\nu_0,\nu_1,\ldots,\nu_{n-1}}\|Q^{(n)}_{n,k}\|_{L_{p'}[0;1]}.
\end{equation}

Таким образом
\begin{equation}\label{eq:Amin}
A_{n,n-1,\infty}(a)=\min_{u \in\mathcal{P}_{n-1}}\|\chi_{[0;a]}-u\|_{L_1[0;1]}
\end{equation}

Справедлива следующая теорема.
\begin{thm}
Рассмотрим функцию
$$
B_n(a)=\tg\frac{\pi}{2(n+1)}\sqrt{a-a^2}.
$$

Точки локальных максимумов функций $A_{n,n-1,\infty}(a)$ на отрезке $[0;1]$ равны $a_j=\sin^2\frac{\pi j}{2(n+1)}$, $j=1,2,\ldots, n$. Значения в этих точках равны
\begin{equation}\label{eq:An1}
A_{n,n-1,\infty}(a_j)=B_n(a_j).
\end{equation}
При четном $n-1$ точкой глобального максимума функции $A_{n,n-1,\infty}$ является точка $a=1/2$, при нечетном $n-1$ точками  глобального максимума функции $A_{n,n-1,\infty}$ являются  ближайшие к $1/2$ точки локального максимума, равные  $\sin^2\frac{\pi n}{4(n+1)}$ и $\sin^2\frac{\pi (n+2)}{4(n+1)}$. Таким образом
\begin{gather*}
\Lambda_{n,n-1,\infty,\infty}=\frac12\tg\frac{\pi}{2(n+1)},\quad \text{если } n\quad \text{нечетно},\\
\Lambda_{n,n-1,\infty,\infty}=\frac12\tg\frac{\pi}{2(n+1)}\sin\frac{\pi n}{2(n+1)},\quad \text{если } n \quad \text{четно}.
\end{gather*}
\end{thm}
\begin{proof}

В работе \cite{D} построены наилучшие приближения в $L_1[-1;1]$ характеристической функции отрезка $\chi_{[-1;a]}$ алгебраическими многочленами. Показано, что эта задача симметрична относительно середины отрезка. С учетом пересчета на отрезок $[0;1]$,  в этой работе показано, что точки $a_j$ являются точками локальных максимумов функции
$$
F(a):=\min_{u\in \mathcal P_{n-1}}\|\chi_{[0;a]}-u\|_{L_1[0;1]}
$$
и
$$
F(a_j)=B_n(a_j),
$$
т.е. функция $B_n(a)$ является огибающей для функции $F$. Из формулы \eqref{eq:Amin} следует, что
$B_n$ являются огибающими для $A_{n,n-1,\infty}$ и
$$
A_{n,n-1,\infty}(a_j)=B_n(a_j),\quad j=1,2,\ldots,n.
$$

При этом из вида $B_n$ и формул для точек $a_j$, $j=1,2,\ldots n$ следует, что при четном $n-1$ $a=1/2$ есть точка глобального максимума функции $A_{n,n-1,\infty}$, а при нечетном $n-1$ точкой глобального максимума функции $A_{n,n-1,\infty}$ является ближайшая к 1/2 точка локального максимума. Осталось заметить, что из симметричности $A_{n,n-1,\infty}$ следует, что при нечетном $n-1$ $a=1/2$ является точкой локального минимума. На рисунке представлены эскизы графиков $B_n$ и $A_{n,n-1,\infty}$ при $n=3,4$.  Алгоритм построения графиков $A_{n,n-1,\infty}$
приведен в \cite{D}.
\end{proof}

\begin{picture}(400,150)
\put(0,0){
\begin{picture}(200,150)
\qbezier[60](10,30)(25,80)(75,80)
\qbezier[60](75,80)(125,80)(140,30)
\thicklines
\linethickness{0.9mm}
\path(10,30)(28,62)
\linethickness{0.3mm}
\qbezier(28,62)(40,59)(75,80)
\qbezier(122,62)(110,59)(75,80)
\linethickness{0.9mm}
\path(122,62)(140,30)
\thinlines
\put(0,30){\vector(1,0){160}}\put(155,22){$x$}
\put(10,10){\vector(0,1){110}}\put(2,109){$y$}
\put(4,20){$0$}
\put(137,20){$1$}
\multiput(75,28)(0,7){8}{\line(0,1){3}}
\put(68,20){$1/2$}
\put(30,05){Графики $A_{3,2}$  и $B_{3,2}$}
\qbezier[7](110,115)(117,115)(124,115)
\put(130,115){${}_{B_{3,2}}$}
\put(110,105){\line(1,0){15}}
\put(130,105){${}_{A_{3,2}}$}
\end{picture}
}
\put(280,0){
\begin{picture}(200,150)
\qbezier[60](10,30)(25,75)(75,75)
\qbezier[60](75,75)(125,75)(140,30)
\thicklines
\linethickness{0.9mm}
\path(10,30)(24,55)
\linethickness{0.3mm}
\qbezier(24,55)(46,53)(55,72)
\qbezier(126,55)(104,53)(95,72)
\qbezier(55,72)(75,53)(95,72)
\linethickness{0.9mm}
\path(126,55)(140,30)
\thinlines
\put(0,30){\vector(1,0){160}}\put(155,22){$x$}
\put(10,10){\vector(0,1){110}}\put(2,109){$y$}
\put(4,20){$0$}
\put(137,20){$1$}
\multiput(75,28)(0,7){7}{\line(0,1){3}}
\put(68,20){$1/2$}
\qbezier[7](110,115)(117,115)(124,115)
\put(130,115){${}_{B_{4,3}}$}
\put(110,105){\line(1,0){15}}
\put(130,105){${}_{A_{4,3}}$}
\put(30,05){Графики $A_{4,3}$  и $B_{4,3}$}
\end{picture}
}
\end{picture}

\section{Точные константы вложения при $p=1$}

Основным результатом этого раздела является следующая теорема.

\begin{thm}\label{tm:1}
$$
\Lambda_{n,k,1,\infty}=\frac12.
$$
\end{thm}
\begin{proof}
Для произвольной функции $y\in\mathring{W}^n_p[0;1]$ и $k=0,1,\ldots,n-1$ выполнены соотношения
\begin{gather}\label{eq:ka}
y^{(k)}(a)=\int_{0}^a \frac{y^{(n)}(t) (x-t)^{n-k-1}}{(n-k-1)!}\,dt,\\
y^{(k)}(a)=(-1)^{n-k}\int_{a}^1 \frac{y^{(n)}(t) (x-t)^{n-k-1}}{(n-k-1)!}\,dt.
\end{gather}
Сложив эти два равенства, получим что
$$
2|y^{(k)}(a)|\leqslant \frac1{(n-k-1)!}\int_0^1 |y^{(n)}(t)|dt=\frac1{(n-k-1)!}\|y^{(n)}\|_{L_1[0;1]},
$$
откуда следует, что
$$
\Lambda_{n,k,1,\infty}\leqslant \frac1{2(n-k-1)!}.
$$
В частности
$$
\Lambda_{n,n-1,1,\infty}\leqslant \frac12.
$$

2) Из представления \eqref{eq:Q} следует, что по переменной $x$ функции
$\left.Q^{(n)}_{n,k}(x,a)\right|_{[0;a)}$ и $\left.Q^{(n)}_{n,k}(x,a)\right|_{(a;1)}$ суть многочлены степени $n-1$ и
$$
Q^{(n)}_{n,n-1}(a-0,a)-Q^{(n)}_{n,n-1}(a+0,a)=1,
$$
откуда следует, что
$$
\|Q^{(n)}_{n,n-1}\|_{L_\infty[0;1]}\geqslant\frac12,
$$
причем для выполнения равенства необходимо, чтобы выполнялось равенство
$$
|Q^{(n)}_{n,n-1}(a-0,a)|=|Q^{(n)}_{n,n-1}(a+0,a)|=\frac12,
$$
при этом значения $Q^{(n)}_{n,n-1}(a-0,a)$ и $Q^{(n)}_{n,n-1}(a+0,a)$ имеют разные знаки.
\end{proof}

\section{Обсуждение полученных результатов}

Полученные результаты для $p=2$ (см. \cite{GarmSh2}) и $p=\infty$ позволяют выдвинуть гипотезу, что при всех $p\in(1;\infty]$ при четных $k$ середина отрезка является точкой глобального максимума для функций $A_{n,k,p}$, а при нечетных $k$ --- точкой локального минимума. При этом, при нечетных $k$ точка глобального максимума функции $A_{n,k,p}$ совпадает с ближайшей к середине отрезка точкой локального максимума.

Авторы признательны А.Г.~Бабенко за полезные обсуждения.


\begin{thebibliography}{99}
\bibitem{Shm} E.Shmidt, \textit{\"{U}ber die Ungleichung, welche die Integrale \"{u}ber eine Potenz einer Function und \"{u}ber eine andere Potenz ihrer Ableitung verbindet.}//Math.Ann, 1940, \textbf{117}, 301--326.

\bibitem{Naz} A.\,I.\,Nazarov, \textit{On exact constant in the generalized Poincar\'{e} inequality}//Journal of Mathematical Sciences, 2002, v.112, No. 1, 4029--4047.

\bibitem{Kalyab} Г.\,А.\,Калябин, \textit{Точные оценки для производных функций из классов Соболева $\mathring{W}^{r}_2(-1;1)$}//Труды МИАН, 2010, т.269,  143--149.

\bibitem{NazM} Е.\,В.\,Мукосеева, А.\,И.\,Назаров, \textit{О симметрии эктремали в некоторых теоремах вложения}//Зап. научн. сем. ПОМИ, 2014, т.425, 35--45.

\bibitem{NazM2} Е.\,В.\,Мукосеева, А.\,И.\,Назаров, \textit{Corrigendum}//Зап. научн. сем. ПОМИ, 2020, т.489, 225.

\bibitem{Garm} Т.А. Гарманова, \textit{Оценки производных в пространствах Соболева в терминах гипергеометрических функций}, Матем. заметки, 2021, т.109, вып.4, 500--507.

\bibitem{GarmSh2} Т.А.Гарманова, И.А.Шейпак, \textit{О точных оценках производных четного порядка в пространствах Соболева }//Функциональный анализ, 2021, т.55, вып.1, 43-–55.

\bibitem{BK} А.\,Г.\,Бабенко, Ю.\,В.\,Крякин, \textit{Интегральное приближение характеристической функции интервала тригонометрическими полиномами}//Труды Ин-та математики и механики УрО РАН, 2008, т.14, № 3,  19-–37.


\bibitem{GarmSh} Т.\,А.\,Гарманова, И.\,А.\,Шейпак, \textit{Явный вид  экстремалей в задаче о  константах вложения в пространствах Соболева}//Труды Московского математического общества, 2019, т.80, вып. 2, 221-246.

\bibitem{D} М.\,В.\,Дейкалова, \textit{Интегральное приближение характеристической функции сферической шапочки алгебраическими многочленами}//Труды Ин-та математики и механики УрО РАН, 2010,
т.16, № 4, 144-–155.


\bibitem{WKNYT}  K.~Watanabe, Y.~Kametaka, A.~Nagai, H.~Yamagishi, K.~Takemura, \textit{Symmetrization of Functions and the Best Constant of 1-DIM Lp Sobolev Inequality}//
Journal of Inequalities and Applications, 2009, v. 2009, 1--12.

\end{thebibliography}
\end{document}